\documentclass[12pt,reqno,a4paper]{article}

\usepackage{hyperref}
\usepackage{amsmath}
\usepackage{amsthm}
\usepackage{amssymb}
\usepackage{dsfont}
\usepackage{stmaryrd}
\usepackage{commath}
\usepackage{mathrsfs}
\usepackage{enumitem}
\usepackage{graphicx}
\font\msbm=msbm10

\theoremstyle{plain}
\newtheorem{theorem}{Theorem}
\newtheorem{lemma}[theorem]{Lemma}
\newtheorem{corollary}[theorem]{Corollary}

\theoremstyle{definition}

\newtheorem{example}[theorem]{Example}
\newtheorem{remark}[theorem]{Remark}
\def\mathbb#1{\hbox{\msbm{#1}}}

\newcommand{\field}[1]{\ensuremath{\mathds{#1}}}

\newcommand{\N}{\field N}

\newcommand{\R}{\field R}

\newcommand{\sign}{{ \ensuremath{\mbox{\rm sign} \,} }}

\newcommand{\eps}{\varepsilon}
\newcommand{\id}{\mathrm{id}}

\newcommand{\be}{\begin{equation}}
\newcommand{\ee}{\end{equation}}
\newcommand{\beq}{\begin{eqnarray}}
\newcommand{\beqq}{\begin{eqnarray*}}
\newcommand{\eeq}{\end{eqnarray}}
\newcommand{\eeqq}{\end{eqnarray*}}

\newcommand{\Sphere}{\ensuremath{\mathds S^{d-1}}}
\newcommand{\Sph}{\ensuremath{\mathds S}}

\usepackage{ifdraft}
\usepackage[colorinlistoftodos, linecolor=red,bordercolor=red, backgroundcolor=red,textsize=small,textwidth=3cm]{todonotes}
\usepackage{xcolor}
\usepackage{setspace}
\usepackage[normalem]{ulem}
\usepackage{comment}
\usepackage{framed}

\ifdraft{\usepackage{showkeys}}{}

\newcounter{todocounter}
\newenvironment{todobox}[3][]
{\addtocounter{todocounter}{1}\todo[caption={\protect\hypertarget{todo\thetodocounter}{}#2 #3},noline, #1]{#2 \hyperlink{todo\thetodocounter}{$\uparrow$}}\color{red}\textbf{#3} }
{\color{black}}

\specialcomment{help}
{
\section*{Help}
}
{}
\ifdraft{}{\excludecomment{help}}

\specialcomment{changes}
{
\section*{Changes}
}
{\listoftodos[TODO's etc.]}
\ifdraft{}{\excludecomment{changes}}

\specialcomment{detail}
{\textcolor{blue!50!black}{/** }\begingroup\color{blue!50!black}}
{\endgroup\textcolor{blue!50!black}{**/} }
\ifdraft{}{\excludecomment{detail}}

\newcommand{\old}[1]{\textcolor{red}{\sout{\ifdraft{#1}{}}}}
\newcommand{\new}[1]{\ifdraft{\textcolor{green!50!black}{#1}}{#1}}

\hypersetup{
   pdfauthor={Aicke Hinrichs, Sebastian Mayer},
   pdfkeywords={entropy numbers, spheres, p-Banach spaces},
   pdftitle={Entropy Numbers of Spheres in Banach and quasi-Banach spaces}
}

\begin{document}

\title{Entropy Numbers of Spheres in Banach and quasi-Banach spaces}

\author{Aicke Hinrichs\\
Institut f\"ur Analysis, Johannes Kepler Universit\"at Linz\\
Altenbergerstra\ss e 69, 4040 Linz, Austria\\
email: aicke.hinrichs@jku.at\\
\qquad
\\
Sebastian Mayer\\
Institut f\"ur numerische Simulation\\
Endenicher Allee 62, 53115 Bonn, Germany\\
email: mayer@ins.uni-bonn.de
}

\maketitle

\begin{abstract}
We prove sharp upper bounds on the entropy numbers $e_k(\Sphere_p,\ell_q^d)$  of the $p$-sphere in $\ell_q^d$ in the case $k \geq d$ and $0< p \leq q \leq \infty$.
In particular, we close a gap left open in recent work of the second author, T. Ullrich and J. Vybiral.
We also investigate generalizations to spheres of general finite-dimensional quasi-Banach spaces.
\end{abstract}

\begin{help}

\begin{itemize}
 \item \new{Piece of text which is new in this version.}
 \item \old{Piece of text that should be removed for some reason.}
 \item \textcolor{blue!50!black}{/** more detailed descriptions; internal comments; $\dots$ **/}
 \item Remove the \texttt{draft} option to hide working draft annotations.
\end{itemize}

\end{help}

\begin{changes}
Changes to draft 3.
\begin{itemize}
\item Added a proof for Corollary \ref{res:sym}.
\item Section \ref{sec:general_spheres} now contains concrete examples (Lorentz and Orlicz norms).
\end{itemize}
\end{changes}

\section{Introduction}

Entropy numbers are a central concept in approximation theory. 
They quantify the compactness of a given set with respect to some reference space. 
For $K$  being a subset of a finite-dimensional (quasi-)Banach space $Y$, the $k$-th (dyadic) entropy number $e_k(K,Y)$ is defined as
\be\label{defi:entropy}
    e_k(K,Y)=\min\Big\{\varepsilon>0:K \subset \bigcup\limits_{j=1}^{2^{k-1}} (x_j+\varepsilon   B_Y)
    \text{ for some } x_1,\dots,x_{2^{k-1}}\in Y\Big\}\,,
\ee
where $B_Y$ denotes the closed unit ball in the space $Y$.
A detailed discussion and historical remarks can be found in the monographs \cite{CaSt90,EdTr96}.

Recall the (quasi-)norms $\|x\|_p := (\sum_{i=1}^d |x_i|^p)^{1/p}$ for $0 < p < \infty$ and $\|x\|_{\infty} := \max_{i=1}^d |x_i|$. 
The entropy numbers $e_k(B_p^d, \ell_q^d)$ of the closed unit balls $B_p^d = \{ x \in \R^d: \|x\|_p \leq 1\}$ in $\ell_q^d = (\R^d,\|\cdot\|_q)$ for $0<p,q\leq \infty$ 
are well-understood for more than a decade now \cite{Ku01,Sch84,Tr97}. For the reader's convenience, we restate the result. Recall that $f \lesssim g$ means that there is a constant $C > 0$ such that $f \leq C g$; the notation $f \asymp g$ means that $g \lesssim f$ and $f \lesssim g$.
\begin{lemma}\label{lem:schuett} Let $0<p\leq q \leq \infty$ and let $k$ and $d$ be natural numbers. Then,
$$
e_k(\bar B_p^d,\ell_q^d) \asymp \left\{
\begin{array}{rcl}1&:&1\leq k \leq \log(d),\\
\Big(\frac{\log(1+d/k)}{k}\Big)^{1/p-1/q}&:&\log(d)\leq k\leq d,\\
2^{-k/d}d^{1/q-1/p}&:&k\geq d\,.
\end{array}
\right.
$$
The constants behind ``$\asymp$'' do neither depend on $k$ nor on $d$. They only depend on the parameters $p$ and $q$.
\end{lemma}
 
Regarding $p$-spheres $\Sphere_p = \{ x \in \R^d: \|x\|_p = 1\}$, one would expect a similar behavior of the entropy numbers $e_k(\Sphere_p,\ell_q^d)$. A rigorous proof, however, has been provided just recently in \cite{MUV13}. Let us restate this result, as well. 
\begin{lemma}\label{res:muv}
Let $d\in \N$, $d\geq 2$, $0<p\leq q \leq \infty$ and $\bar p = \min\{1,p\}$. Then,
\begin{enumerate}[label=(\roman*)]
\item \[ 2^{-k/(d-1)} d^{1/q-1/p} \; \lesssim \; e_k(\Sphere_p, \ell_q^d) \; \lesssim \; 2^{-k/(d-\bar p)} d^{1/q-1/p} \quad \text{for } k  \geq d. \]
\item
$$
e_k(\Sphere_p,\ell_q^d) \asymp \left\{
\begin{array}{ccl}1&\text{for}&1\leq k \leq \log d,\\
\Big(\frac{\log(1+d/k)}{k}\Big)^{1/p-1/q}&\text{for}&\log d\leq k\leq d.
\end{array}
\right.
$$
\end{enumerate}
\end{lemma}
The proof largely mimics the corresponding proof for $p$-unit balls, including the volume arguments which are utilized in the case $k \geq d$. Once $p<1$ however, the volume arguments get too coarse for the upper bound. They do no longer lead to the correct order of decay; instead of $2^{-k/(d-1)}$, we only see $2^{-k/(d-p)}$.

In this work, we present an alternative proof for the case $k \geq d$ which gives the correct order of decay $1/(d-1)$. In essence, the proof relies on a scheme to construct a covering of $\Sphere_p$ from a covering of $B_p^{d-1}$. In Section \ref{sec:p-spheres}, we prove 
\begin{theorem} \label{thm:p-sphere}
Let $d\in \N$, $d\geq 2$, $0<p\leq q \leq \infty$. Then,
\[  e_k(\Sphere_p, \ell_q^d) \; \asymp \; 2^{-k/(d-1)} d^{1/q-1/p}, \quad \text{for } k  \geq d.  \]
\end{theorem}

Finally in Section \ref{sec:general_spheres}, we generalize the methods presented in Section \ref{sec:p-spheres} to study entropy numbers $e_k(\Sph_X,Y)$ for general finite-dimensional quasi-Banach spaces $X$ and $Y$.

\paragraph{Notations}
For $e \in \{-1,1\}^d$ a sign vector, we denote the associated orthant by $Q_e :=\big\{ x \in \R^d: 0 \leq \sign e_j \, x_j \leq 1 \big\}$. Furthermore, by $H_i$ we denote the hyperplane where the $i$-th coordinate of every vector equals $0$ and  $C_i = \{ x \in \R^d: \; |x_i| \leq |x_j| \text{ for } j=1,\dots,d \}$. Whenever we write $|x|\leq |y|$, then this is meant  coordinate-wise: $|x_j| \leq |y_j|$ for each $j=1,\dots,d$.

\section{Entropy numbers of $p$-spheres}\label{sec:p-spheres}
Let us start this section with a simple proof that the entropy numbers of the $p$-sphere $\Sphere_p$ are always bounded from below by the respective entropy numbers of the $p$-unit ball $  B_p^{d-1}$.
The known behaviour of the latter immediately implies all the lower bounds in Lemma \ref{res:muv} and Theorem \ref{thm:p-sphere}.

\begin{theorem}
\label{thm:lbp}
Let $0<p,q \leq \infty$. For all natural numbers $k$, we have
\[
 e_k(\Sphere_p, \ell_q^d) \geq e_k(  B_p^{d-1}, \ell_q^{d-1}).
\]
\end{theorem}

\begin{proof}
Let $\varepsilon=e_k(\Sphere_p, \ell_q^d)$ and choose an associated covering
\[
 \Sphere_p \subset \bigcup\limits_{j=1}^{2^{k-1}} (x_j+\varepsilon   B_q^d).
\]
Let $P_1$ be the projection in $\R^d$ onto the hyperplane $H_1$ setting the first coordinate to zero.
Now, identify $B_p^{d-1}$ and  $B_q^{d-1}$ with its natural isometric embedding into $H_1$.
Then
\[ P_1(B_q^d)=B_q^{d-1} \qquad \text{and} \qquad P_1( \Sphere_p) = B_p^{d-1} \]
and the linearity of $P_1$ imply
\[
 B_p^{d-1} =       P_1( \Sphere_p) 
           \subset \bigcup\limits_{j=1}^{2^{k-1}} \big(P_1 x_j+\varepsilon   P_1(B_q^d) \big)
           =       \bigcup\limits_{j=1}^{2^{k-1}} \big(P_1 x_j+\varepsilon   B_q^{d-1} \big).
\]
By definition of the entropy numbers this gives  $e_k(  B_p^{d-1}, \ell_q^{d-1}) \le \varepsilon$ and proves the theorem.
\end{proof}

The upper bounds in $(ii)$ of Lemma \ref{res:muv} directly follow from the
corresponding upper bounds for $B_p^d$ and the inclusion $\Sphere_p \subset B_p^d$.
We turn to the proof of the upper bound on the entropy numbers of $\Sphere_p$ for $k\ge d$ in Theorem \ref{thm:p-sphere}. 
The covering construction is based on the bijective shifting map
\begin{align*}
 \Delta_e^p: \quad   B_p^d \cap Q_e \cap H_i & \to \Sphere_p \cap Q_e \cap C_i,\\
  x & \mapsto x + s(x) e,
\end{align*}
which shifts $x$ by $0 \leq s(x) \leq 1$ along the diagonal $e$ until it hits the $p$-sphere. Since the vectors in the domain of $\Delta^p_e$ all have the $i$-th  coordinate equal to $0$, the shift $s(x)$ corresponds to the $i$-th  coordinate of $\Delta_e^p(x)$. Below we provide bounds on the difference between two shifts $s(x)$, $s(y)$ that depend on the $\ell_\infty$-distance of $x$ and $y$.

\begin{lemma}\label{res:Lip2}
Let $x,y \in   B_p^d \cap Q_e \cap H_i$ such that $|x| \leq |y|$. Then
\[
 s(y) \le s(x) \le s(y)+\|x-y\|_{\infty}.
\]
\end{lemma}
\begin{proof}
We give the argument for $p<\infty$, the proof is easily adapted for $p=\infty$.
By symmetry, it is enough to check the claim for the positive orthant associated with $e=(1,\dots,1)$ and $i=d$.
Then $x_d=y_d=0$ and the assumption $|x|\le|y|$ translates into $0 \le x_j\le y_j$ for $j=1,\dots,d-1$.
Moreover, $\sigma=s(x)\ge 0$ and $\tau=s(y)\ge 0$ are given as the unique nonnegative solutions of the equations
\[
 \sum_{j=1}^{d-1} (x_j+\sigma)^p + \sigma^p = 1 \qquad \text{and} \qquad 
 \sum_{j=1}^{d-1} (y_j+\tau)^p + \tau^p = 1.
\] 
Now $\tau\le\sigma$ directly follows from $0\le x_j\le y_j$ for $j=1,\dots,d-1$.
The remaining inequality $\sigma\le \|x-y\|_{\infty} + \tau$ is a consequence of
\begin{eqnarray*}
  \sum_{j=1}^{d-1} (x_j+\sigma)^p + \sigma^p 
	  &=&    \sum_{j=1}^{d-1} (y_j+\tau)^p + \tau^p \\
    &\le&  \sum_{j=1}^{d-1} \big(x_j+(y_j-x_j) + \tau\big)^p + \tau^p \\
		&\le&  \sum_{j=1}^{d-1} \big(x_j+\|x-y\|_{\infty} + \tau\big)^p + \big( \|x-y\|_{\infty} + \tau \big)^p.
\end{eqnarray*}
%
\end{proof}

Now it is easy to establish the following for $\Delta^p_e$:

\begin{lemma} \label{lem:xx}
For any $x,y \in   B_p^d \cap Q_e \cap H_i$, we have
\[
 \|\Delta^p_e(x) - \Delta^p_e(y)\|_{\infty} \leq 2\|x-y\|_{\infty}.
\]
\end{lemma}
\begin{proof}
For the moment, assume that $|x|\leq |y|$. The previous lemma and $|x|\leq |y|$ immediately give 
\[
 -\|x-y\|_{\infty} \leq |x_j|-|y_j| +s(x) - s(y) \leq \|x-y\|_{\infty}
\]
for every $j=1,\dots,d$.
Hence, we have $\|\Delta_e^p(x) - \Delta_e^p(y)\|_{\infty} \leq \|x-y\|_{\infty}$.

Now for arbitrary $x,y \in   B_p^d \cap Q_e \cap H_i$, let $z \in B_p^d \cap Q_e \cap H_i$ be the vector given by $z_j = \sign e_j \min(|x_j|,|y_j|)$ for $j=1,\dots,d$. But for this vector, both $|z|\leq|x|$ and $|z|\leq|y|$ hold true. Hence, a simple application of the triangle inequality yields $\|\Delta_e^p(x) - \Delta_e^p(y)\|_{\infty} \leq 2\|x-y\|_{\infty}$.
\end{proof}

\begin{theorem}\label{res:entropy_sphere}
Let $0<p,q\le \infty$ and $d \geq 2$. For $k \geq d$ it holds true that
\[
 e_k(\Sphere_p,\ell_q^d) \lesssim 2^{-\frac{k}{d-1}} d^{1/q-1/p}.
\] 
\end{theorem}

\begin{proof}
We first treat the case $q=\infty$.
Let $\mathcal N_{e,i}(\varepsilon)$ be a minimal $\varepsilon$-covering of $  B_p^d \cap Q_e \cap H_i$ in $\ell_{\infty}^d$. 
Since $  B_p^d \cap Q_e \cap H_i$ is a subset of a $(d-1)$-dimensional $\ell_p$-ball, we have $$ \big| \mathcal N_{e,i}(\varepsilon) \big| \le N_{\varepsilon}(  B_p^{d-1},\ell^{d-1}_{\infty}).$$
Lemma \ref{lem:xx} implies that  the set $\Delta^p_e(\mathcal N_{e,i}(\varepsilon))$ is a $2\varepsilon$-covering of $\Sphere_p \cap Q_e \cap C_i$. 
Consequently,
\[
 \tilde{ \mathcal N} (2 \varepsilon) := \bigcup_{e \in \{-1,1\}^d} \bigcup_{1\le i\le d} \Delta_e^p(\mathcal N_{e,i}(\varepsilon))
\]
is a $2\varepsilon$-covering of $\Sphere_p$. 
Moreover 
\[ 
  \big|\tilde{\mathcal N}(2\varepsilon)\big| \leq \sum_{e \in \{-1,1\}^d} \sum_{1\le i\le d}  \big| \mathcal N_{e,i}(\varepsilon) \big| \leq 2^d \, d \, N_{\varepsilon}(  B_p^{d-1},\ell^{d-1}_{\infty}).
\] 
By definition of entropy numbers,
\[
 e_k(\Sphere_p,\ell_{\infty}^d) \leq 2 e_{k-d-\lceil \log d \rceil}(  B^{d-1}_p, \ell_{\infty}^{d-1}).
\]
Hence the assertion follows for $k \geq 2d + \lceil \log d \rceil$ immediately from the upper bound in the third case of Lemma \ref{lem:schuett}. For $d \leq k < 2d + \lceil \log d \rceil$, we have by monotonicity of entropy numbers and Lemma \ref{res:muv}, (ii) that $e_k(\Sphere_p,\ell_{\infty}^d) \leq e_d(\Sphere_p,\ell_\infty^d) \lesssim d^{-1/p}$; now, due to the assumption $k < 2d + \lceil \log d \rceil$ we may estimate $2^{-6} \leq 2^{-k/(d-1)}$ and hence the assertion holds true for all $k \geq d$.

The case of general $q$ is a consequence of the factorization property of entropy numbers, see for instance \cite[Section 1.3]{EdTr96},
\[
 e_k(\Sphere_p,\ell_q^d) \leq \| \id: \ell_\infty^d \to \ell_q^d \| \, e_k(\Sphere_p,\ell_{\infty}^d) = d^{1/q} \, e_k(\Sphere_p,\ell_{\infty}^d).
\] 
\end{proof}

\begin{remark}
For $p=q$, the \emph{Mazur map}
$$M_p:\Sphere_2 \to \Sphere_p, \; x \mapsto ((\sign x_j) |x_j|^{2/p})_{i=1}^d$$
provides a different technique to derive matching bounds for $p<1$ in Lemma \ref{res:muv}.
This mapping forms a homeomorphism between the $\ell_2^d$- and the $\ell_p^d$-sphere.
In particular, for $0<p<2$ the Mazur map is Lipschitz with constant $L=L(p)$ independent of the dimension, see \cite{Wes92}.
Hence, any $\eps$-covering of the $2$-sphere yields an $L\eps$-covering of the $p$-sphere by mapping the centers to $\Sphere_p$ with the Mazur map.
This shows that
$$ e_k(\Sphere_p,\ell_p) \le L  e_k(\Sphere_2,\ell_2), $$
which allows to transfer the upper bounds for the case $p=2$ in Lemma \ref{res:muv} to the case $0<p<1$.
\end{remark}

\section{Entropy numbers of spheres in general quasi-Ba\-nach spaces}\label{sec:general_spheres}
In this section, we elaborate on the methods presented in the previous section, generalizing them to arbitrary finite-dimensional quasi-Banach spaces. To this end, let $X$ be $\R^d$ equipped with some quasi-norm $\| \cdot  \|$.
Let $\Sph_X$ be the unit sphere in $X$. 
Let $X_i$ be the hyperplane $H_i$ considered as a subspace of $X$, i.e. its unit ball is
$B_X \cap H_i$.
Let $X^i$ be the hyperplane $H_i$ with the quasi-norm whose unit ball is the
image of $B_X$ under the orthogonal projection of $\R^d$ onto $H_i$. 
Observe that the quasi-norms of $X_i$ and $X^i$ coincide for $X=\ell_p^d$, $0<p\le\infty$.
That $X_i$ and $X^i$ are the same is always true if the orthogonal projection onto $H_i$ is a norm 1 projection, 
which is the case, in particular,  if $X$ is a Banach space. 

The following theorem provides a generalization of Theorem \ref{thm:lbp}.

\begin{theorem}
\label{thm:lbX}
Let $X$ and $Y$ be $\R^d$ with quasi-norms as above. 
For all natural numbers $k$, we have
\[
 e_k(\Sph_X, Y) \geq \max_{1 \le i \le d} e_k(  B_{X^i}, Y^i).
\]
\end{theorem}

\begin{proof}
The proof of Theorem \ref{thm:lbp} can be easily adapted to this situation. 
Let $\varepsilon=e_k(\Sph_X, Y)$ and choose an associated covering
\[
 \Sph_X \subset \bigcup\limits_{j=1}^{2^{k-1}} (x_j+\varepsilon   B_Y).
\]
Let $P_i$ be the orthogonal projection in $\R^d$ onto the hyperplane $H_i$ setting the first coordinate to zero.
Then we have
\[ P_i(B_Y)=B_{Y^i} \qquad \text{and} \qquad P_i( \Sph_X) = B_{X^i}. \]
Since $P_i$ is linear this implies
\[
 B_{X^i} =       P_i( \Sph_X) 
           \subset \bigcup\limits_{j=1}^{2^{k-1}} \big(P_i x_j+\varepsilon   P_i(B_Y) \big)
           =       \bigcup\limits_{j=1}^{2^{k-1}} \big(P_i x_j+\varepsilon   B_{Y^i} \big).
\]
By definition of the entropy numbers this gives  $e_k(  B_{X^i}, Y^i) \le \varepsilon$ and proves the theorem.
\end{proof}

To generalize the upper bounds, we have to assume a monotonicity property of the quasi-norm of $X$ in each orthant $Q_e$.
Let us call $X$ {\em monotone} if for any $e \in \{ -1,1\}^d$ and for $x,y \in Q_e$ with $|x|\le |y|$ we have $\|x\| \le \|y\|$.
Again, we define a bijective shifting map by
\begin{align*}
 \Delta_e^X: \quad   B_X \cap Q_e \cap H_i & \to \Sph_X \cap Q_e \cap C_i,\\
  x & \mapsto x + s(x) e,
\end{align*}
which shifts $x$ by $0 \leq s(x) \leq 1$ along the diagonal $e$ until it hits the sphere of $X$.
Then analogues of Lemmas \ref{res:Lip2} and \ref{lem:xx} hold true.

\begin{lemma} 
Let $x,y \in   X \cap Q_e \cap H_i$ such that $|x| \leq |y|$. Then
\[
 s(y) \le s(x) \le s(y)+\|x-y\|_{\infty}.
\]
\end{lemma}

\begin{lemma} \label{lem:xxx}
For any $x,y \in   X \cap Q_e \cap H_i$, we have
\[
 \|\Delta_e^X(x) - \Delta_e^X(y)\|_{\infty} \leq 2\|x-y\|_{\infty}.
\]
\end{lemma}

We leave the modifications of the proofs to the attentive reader. Now we obtain the following generalization of Theorem \ref{res:entropy_sphere}.
 
\begin{theorem}\label{thm:ubX}
Let $X$ and $Y$ be $\R^d$ with quasi-norms as above and assume that $X$ is monotone.
Let $k \geq d-\lceil \log d \rceil$. Then
\[
 e_k(\Sph_X,Y) \leq 2 \, \| Id: Y \to \ell_\infty^d \| \, \max_{1\le i \le d} e_{k-d-\lceil \log d \rceil}(  B_{X_i}, \ell_{\infty}^{d-1}).
\] 
\end{theorem}

\begin{proof}
At first, we treat the case $Y= \ell_{\infty}^d$.
Let $\mathcal N_{e,i}(\varepsilon)$ be a minimal $\varepsilon$-covering of $  B_X \cap Q_e \cap H_i$ in $\ell_{\infty}^d$. 
Since $ B_X \cap Q_e \cap H_i$ is a subset of a $B_{X_i}$, we have
$$ \big| \mathcal N_{e,i}(\varepsilon) \big| \le N_{\varepsilon}(  B_{X_i},\ell^{d-1}_{\infty}),$$
where $N_{\varepsilon}(  B_{X_i},\ell^{d-1}_{\infty})$ denotes the \emph{covering number}, that is the smallest possible number of balls $\varepsilon B_\infty^d$ necessary to cover $B_{X_i}$.
Lemma \ref{lem:xxx} implies that  the set $\Delta^X_e(\mathcal N_{e,i}(\varepsilon))$ is a $2\varepsilon$-covering of $\Sph_X \cap Q_e \cap C_i$. 
Consequently,
\[
 \tilde{ \mathcal N} (2\varepsilon) := \bigcup_{e \in \{-1,1\}^d} \bigcup_{1\le i\le d} \Delta_e^X(\mathcal N_{e,i}(\varepsilon))
\]
is a $2\varepsilon$-covering of $\Sph_X$. 
Moreover 
\[ 
  \big|\tilde{\mathcal N}(2\varepsilon)\big| \leq \sum_{e \in \{-1,1\}^d} \sum_{1\le i\le d}  \big| \mathcal N_{e,i}(\varepsilon) \big| \leq 2^d \, \sum_{i=1}^d \, N_{\varepsilon}(  B_{X_i},\ell^{d-1}_{\infty})
	\leq 2^d \, d \, \max_{1\le i \le d} \, N_{\varepsilon}(  B_{X_i},\ell^{d-1}_{\infty}).
\] 
By definition of entropy numbers,
\[
 e_k(\Sph_X,\ell_{\infty}^d) \leq 2 \, \max_{1\le i \le d} \, e_{k-d-\lceil \log d \rceil}(  B_{X_i}, \ell_{\infty}^{d-1}).
\]

The case of general $Y$ is a consequence of the factorization property of entropy numbers:
\[
 e_k(\Sph_X,Y) \leq \|  \id: \ell_\infty^d \to Y \| \,  e_k(\Sph_X,\ell_{\infty}^{d-1}).
\] 
\end{proof}

To conclude this note, let us discuss the situation that both $X$ and $Y$ are Banach spaces with norms being symmetric with respect to the canonical basis. Then it is possible to make the bounds in the Theorems \ref{thm:lbX} and \ref{thm:ubX} concrete by applying the results of Sch\"utt \cite{Sch84}. We call a Banach space $X$ with norm $\|\cdot\|$ \emph{symmetric} if the canoncial basis $\{e_1,\dots,e_d\}$ has the following property. For all permutations $\pi$, all sings $\varepsilon_i$, and all $x_i \in \R$, we have
\[ 
 \Big\|\sum_{i=1}^d \varepsilon_i x_{\pi(i)} e_i\Big\| = \Big\|\sum_{i=1}^d x_i e_i\Big\|.
\]
Below we use the notation
\[ 
 \lambda_E(k) = \Big\| \sum_{i=1}^k e_i \Big\|
\]
where $E$ denotes a $d$-dimensional, symmetric Banach space and $\{e_1,\dots,e_d\}$ its canonical basis.

\begin{corollary}\label{res:sym}
Let $X$ and $Y$ be symmetric Banach spaces and assume the canonical basis in both spaces to be normalized. Then, we have
\[ 
 e_k(\Sphere_X,Y) \geq \frac{1}{2e} \; \max_{\ell=k,\dots,d-1} \frac{\lambda_Y(\ell)}{\lambda_X(\ell)}
\]
for $k \leq d-1$ and
\[
 \frac{1}{e} \; 2^{-k/(d-1)} \; \frac{\lambda_Y(d-1)}{\lambda_X(d-1)} \leq e_k(\Sphere_X,Y) \leq 32c \; 2^{-k/(d-1)} \; \frac{\lambda_Y(d-1)}{\lambda_X(d-1)}
\]
for $k \geq 2d+\lceil \log d \rceil - 1$.
\end{corollary}
\begin{proof}
Since $X$ is symmetric, so is $X_i$. Moreover, $X_i$ and $X_j$ are isometrically isomorphic for $i \neq j$ by the symmetry of $X$. 
The same holds true for $Y$. Hence, $e_k(B_{X_i},Y_i) = e_k(B_{X_j},Y_j)$ for $i \neq j$ and all $k$. 
By virtue of \cite[Theorem 5 (1)]{Sch84} we know that
\[ 
 e_k(B_{X_1},Y_1) \geq \frac{1}{2e} \; \max_{\ell=k,\dots,d-1} \frac{\lambda_{Y_1}(\ell)}{\lambda_{X_1}(\ell)}.
\]
Combining with Theorem \ref{thm:lbX} and noting that $\lambda_{X_1}(\ell) = \lambda_{X}(\ell)$ as well as $\lambda_{Y_1}(\ell) = \lambda_Y(\ell)$ for $\ell=1,\dots,d-1$, 
we obtain the stated lower bound for $k \leq d-1$.

Regarding the bounds for $k \geq 2d + \lceil \log d \rceil - 1$, we get from \cite[Theorem 5 (2)]{Sch84} that
\begin{equation}\label{eq:schuett1} 
 \frac{1}{e} \; 2^{-k/(d-1)} \; \frac{\lambda_{Y_1}(d-1)}{\lambda_{X_1}(d-1)} \leq e_k(B_{X_1},Y_1) \quad \text{ for } k \geq d-1
\end{equation}
as well as
\begin{equation}\label{eq:schuett2}
 e_m(B_{X_1},\ell_{\infty}^{d-1}) \leq c \; 2^{-m/(d-1)} \; \frac{\lambda_{\ell_{\infty}^{d-1}}(d-1)}{\lambda_{X_1}(d-1)} \quad \text{ for } m \geq d-1.
\end{equation}
Equation \eqref{eq:schuett1} and Theorem \ref{thm:lbX} immediately give the lower bound. For the upper bound, put $m = k - d - \lceil \log d \rceil$. Noting that $\|\id: Y \to \ell_{\infty}^d\| = \lambda_Y(d) \leq 2 \lambda_Y(d-1)$ and $\lambda_{\ell_{\infty}^{d-1}}(d-1)=1$, Theorem \ref{thm:ubX} combined with \eqref{eq:schuett2} then gives
\[ 
  e_k(\Sphere_X,Y) \leq 4 c \; 2^{(d+ \lceil \log d \rceil)/(d-1)} \; 2^{-k/(d-1)} \; \frac{\lambda_Y(d-1)}{\lambda_X(d-1)} \quad \text{ for } k \geq 2d + \lceil \log d \rceil - 1.
\]
Since we have assumed that $d \geq 2$, we can estimate $2^{(d+ \lceil \log d \rceil)/(d-1)} \leq 8$.
\end{proof}

\begin{remark}
In contrast to Theorem \ref{res:entropy_sphere}, the statement of Corollary \ref{res:sym} cannot be extended to $k \geq d$ in a simple way as we do not know how $e_d(\Sphere_X,Y)$ behaves.
\end{remark}

\begin{example}[Lorentz norms]
For $x \in \R^d$, let $x^* = (x_1^*,\dots,x_n^*)$ denote the non-increasing rearrangement of $x$. We consider \emph{generalized Lorentz norms}
\[ 
 \|x\|_{w,q} := \Big( \sum_{i=1}^d (w(i) |x_i^*|)^q\Big)^{1/q}
\]
for $0<q\leq \infty$ and a non-increasing function $w: [1,\infty) \to (0, \infty)$  
with $w(1) = 1$, $\lim_{t \to \infty} w(t) = 0$ and $\sum_{i=1}^{\infty} w(i) = \infty$. 
Then
\[ 
 \lambda_{w,q}(k) := \Big\|\sum_{i=1}^k e_i \Big\|_{w,q} = \Big(\sum_{i=1}^k w(i)^q\Big)^{1/q}.
\]
Particularly, for the choice $w(t) = t^{1/p - 1/q}$ for $0<q<p<\infty$ we obtain the standard Lorentz norm 
\[
 \|x\|_{w,q} = \|x\|_{p,q} = \big(\sum_{i=1}^d i^{q/p-1} |x_i^*|^q\big)^{1/q}.
\] 
A simple calculation, where one approximates the sum by an integral, shows that $\lambda_{p,q}(k) \asymp k^{1/p}$. So, if we consider $X=\ell_{p,q}^d$ and $Y = \ell_r^d$, we obtain from Corollary \ref{res:sym} that
\[ 
 e_k(\Sphere_{\ell_{p,q}},\ell_r^d) \asymp \frac{2^{-k/(d-1)}}{d^{1/p-1/r}} \asymp e_k(\Sphere_{\ell_p^d},\ell_r^d)
\]
for $k \geq 2d+\lceil \log d \rceil - 1$.
\end{example}

\begin{example}[Orlicz norms]
A convex function $M: \R_+ \to \R_+$ with $M(0)$ and $M(t) > 0$ for $t \neq 0$ is called an \emph{Orlicz function}, which we associate with the norm $\|\cdot\|_M$  
given by
\[ 
 \|x\|_M := \inf \{ \rho > 0: \sum_{i=1}^d M(|x_i|/\rho) \leq 1  \}.
\]
It is easy to calculate that
\[ 
 \lambda_M(k) := \lambda_{(\R^d,\|\cdot\|_M)}(k) = \frac{1}{M^{-1}(1/k)},
\]
where $M^{-1}$ is the inverse function of $M$. Let us consider some specific examples. In the following we always choose $Y=\ell_q^d$ for some $0<q\leq\infty$.

(i) Let $M(t) = \exp(-1/t^2)$ for $0\leq t < 1/2$. See \cite{GoEtAl02} for an application where the associated Orlicz norm appears naturally. 
We have $M^{-1}(t) = \sqrt{-\ln t}$ and consequently
 \[ 
  e_k(\Sphere_M,\ell_q^d) \asymp \sqrt{\ln d}\;  d^{1/q} \;  2^{-k/(d-1)}.
 \]

(ii) The following Orlicz function is taken from \cite{LinTza}. For $p > 1$ and $\alpha > 0$, let $M$ on $[0,t_0)$ be given by $M(t) = t^p \ln(1/t)^{\alpha}$. Here, $t_0 = 1/\exp(\beta + \sqrt{\beta^2 - \gamma})$ with $\beta = \alpha (2p-1)/(p^2-p)$ and $\gamma = (\alpha^2-\alpha)/(p^2-p)$. There is a $t_0' \leq t_0$ such that
 \[ 
  M^{-1}(y) \asymp_{p,\alpha} \left( \frac{y}{\log(1/y)^{\alpha}} \right)^{1/p} \qquad \text{ for } y < M(t_0').
 \]
 Hence, for $k$ sufficiently large, we find
 \[ 
  e_k(\Sphere_X, \ell_q^d) \asymp  \log(d)^{\alpha/p} \; d^{1/q-1/p}  \; 2^{-k/(d-1)}.
 \]
\end{example}

\end{document}